\documentclass{amsart}
\usepackage{amsmath}%
\usepackage{amsfonts}%
\usepackage{amssymb}%
\usepackage{graphicx}%
\usepackage{mathtools}%
\usepackage{xcolor}%
\usepackage{nameref,hyperref,url}

\newtheorem{theorem}{Theorem}
\theoremstyle{plain}

\newtheorem{corollary}{Corollary}

\newtheorem{definition}{Definition}

\newtheorem{lemma}{Lemma}

\newtheorem{proposition}{Proposition}
\newtheorem{remark}{Remark}

\numberwithin{equation}{section}
\newcommand{\vertiii}[1]{{\left\vert\kern-0.25ex\left\vert\kern-0.25ex\left\vert #1 
    \right\vert\kern-0.25ex\right\vert\kern-0.25ex\right\vert}}
\newcommand{\vertiiin}[1]{{\vert\kern-0.25ex\vert\kern-0.25ex\vert #1 
    \vert\kern-0.25ex\vert\kern-0.25ex\vert}}

\begin{document}
\title[Spectral stability for $(-\Delta_p)_{\delta}^s$]{Spectral stability for the peridynamic fractional $p\,$-Laplacian}
\author[Jos\'e C. Bellido]{Jos\'e C. Bellido}
\address{Jos\'e C. Bellido\hfill\break\indent
E.T.S.I. Industriales \& INEI, Departamento de Matem\'aticas \hfill\break\indent 
Universidad de Castilla-La Mancha\hfill\break\indent 
Ciudad Real, E-13071 \hfill\break\indent 
Spain \hfill\break\indent } 
\email{josecarlos.bellido@uclm.es}

\author[Alejandro Ortega]{Alejandro Ortega}
\address{Alejandro Ortega\hfill\break\indent
Departamento de Matem\'aticas\hfill\break\indent 
Universidad Carlos III de Madrid\hfill\break\indent 
Av. de la Universidad 30 \hfill\break\indent 
28911 Legan\'es (Madrid), SPAIN \hfill\break\indent } 
\email{alortega@math.uc3m.es}

\thanks{This work was carried out while the second author was a postdoctoral researcher at Universidad de Castilla-La Mancha in Ciudad Real funded by project  MTM2017-83740-P of the {\it Agencia Estatal de Investigaci\'on, Ministerio de Ciencia e Innovaci\'on} (Spain), which supports this investigation.}
\date{\today}
\subjclass[2010]{Primary 35R11, 35P30, 49J45; Secondary 49J35, 45G05, 47G20} %
\keywords{Fractional $p\,$-Laplacian,  Nonlinear Eigenvalue Problems, Nonlocal Elliptic Problems, $\Gamma$-Convergence, Peridynamics, Volume Constrained Problems}%

\begin{abstract} 
In this work we analyze the behavior of the spectrum of the peridynamic fractional $p\,$-Laplacian, $(-\Delta_p)_{\delta}^{s}$, under the limit process $\delta\to0^+$ or $\delta\to+\infty$. We prove spectral convergence to the classical $p\,$-Laplacian under a suitable scaling as $\delta\to 0^+$ and to the fractional $p\,$-Laplacian as $\delta \to +\infty$.
\end{abstract}
\maketitle


\section{Introduction}
Given a {\it horizon} $\delta>0$, $s\in(0,1)$ and $p\in(1,+\infty)$, we study the spectral stability of the \textit{peridynamic fractional $p\,$-Laplacian}, denoted by $(-\Delta_p)_{\delta}^s$ and defined as
\begin{equation}\label{defpLaplacian}
(-\Delta_p)_{\delta}^su(x)=2P.V.\int_{B(x,\delta)}\frac{|u(x)-u(y)|^{p-2}(u(x)-u(y))}{|x-y|^{n+ps}}dy,
\end{equation}
where P.V. stands for principal value. The {\it peridynamic} denomination is inspired by {\it Peridynamic Theory in Solid Mechanics} \cite{Si} (see also \cite{BellCorPed, DeliaGun}). 
In particular, we address the spectral stability under the limits $\delta\to0^+$ or $\delta\to+\infty$ of the nonlocal nonlinear eigenvalue problem
\begin{equation}\label{eigenproblem}
     \left\{\begin{array}{rl}
     (-\Delta_p)_{\delta}^su=\lambda|u|^{p-2}u &\quad\mbox{in}\quad \Omega,\\
                         u=0\mkern+59mu &\quad\mbox{on}\quad \partial_{\delta}\Omega,\\
		 \end{array}\right.
				\tag{$EP_{\delta}^{s,p}$}
\end{equation}
where $p\in(1,+\infty)$, $s\in(0,1)$, $\Omega\subset\mathbb{R}^N$ is a bounded domain with Lipschitz boundary and $\partial_\delta\Omega$, the nonlocal boundary, is given by 
\begin{equation*}
\partial_{\delta}\Omega=\{y\in\mathbb{R}^N\backslash\Omega: |x-y|<\delta\ \text{for}\ x\in\Omega\}.
\end{equation*}

Study of both local and nonlocal nonlinear eigenvalue problems is a delicate issue considerably more difficult than its linear counterpart. For the local nonlinear eigenvalue problem 
\begin{equation}\label{eigenproblem0}
     \left\{\begin{array}{rl}
     -\Delta_pu=\lambda|u|^{p-2}u &\quad\mbox{in}\quad \Omega,\\
                         u=0\mkern+59mu &\quad\mbox{on}\quad\! \partial\Omega,\\
		 \end{array}\right.
				\tag{$EP_{0}^{1,p}$}
\end{equation}
associated to the classical $p\,$-Laplace operator, $-\Delta_p=-div(|\nabla u|^{p-2}\nabla u)$, it was proved, cf. \cite{Azorero1987}, that \eqref{eigenproblem0} admits a countable set of eigenvalues $\{\lambda_{k}^{0,1,p}\}_{k\in\mathbb{N}}$ such that
\begin{equation*}
0<\lambda_{1}^{0,1,p}<\lambda_{2}^{0,1,p}\leq\ldots\leq\lambda_{k}^{0,1,p}\to+\infty\quad \text{as}\ k\to+\infty.
\end{equation*}
Such a sequence of variational eigenvalues is obtained using a minimax variational principle based on certain topological index as follows. Let us define
\begin{equation*}
\mathcal{S}^{1,p}(\Omega)\vcentcolon=\{u\in W_0^{1,p}(\Omega): \|u\|_{L^p(\Omega)}=1\}
\end{equation*}
and
\begin{equation*}
\mathcal{H}_k^{1,p}\vcentcolon=\left\{A\subset\mathcal{S}^{1,p}(\Omega): A\text{ symmetric and compact, }i(A)\geq k\right\},
\end{equation*}
where $i(A)$ denotes the \textit{Krasnosel'skii genus} of $A$,
\begin{equation*}
i(A)=\inf\left\{k\in\mathbb{N}:\ \exists \text{ a continuous odd map }f: A\mapsto\mathbb{S}^{k-1} \right\},
\end{equation*}
with the convention $i(A)=+\infty$ if no such integer exists and $\mathbb{S}^{N}$ the $N$-dimensional unitary sphere. Then, the nonlinear eigenvalues $\{\lambda_{k}^{0,1,p}\}_{k\in\mathbb{N}}$ are given as minimax values,
\begin{equation}\label{auto}
\lambda_{k}^{0,1,p}=\inf\limits_{A\in\mathcal{H}_k^{1,p}}\max\limits_{u\in A}\|\nabla u\|_{L^p(\Omega)}^p.
\end{equation}
In particular, the first eigenvalue is given as a global minimum,
\begin{equation*}
\lambda_{1}^{0,1,p}=\min\limits_{u\in\mathcal{S}^{1,p}(\Omega)}\|\nabla u\|_{L^p(\Omega)}^p,
\end{equation*}
while the second eigenvalue is given, cf. \cite{Cuesta1999}, as a mountain pass-type minimax value,
\begin{equation*}
\lambda_{2}^{0,1,p}=\inf\limits_{\gamma\in\Gamma(-u_1,u_1)}\max\limits_{u\in\gamma([0,1])}\|\nabla u\|_{L^p(\Omega)}^p,
\end{equation*}
where $u_1$ is the minimizer associated to $\lambda_{1}^{0,1,p}$ and $\Gamma(-u_1,u_1)$ is the set of continuous paths on $\mathcal{S}^{1,p}(\Omega)$ connecting $-u_1$ and $u_1$.\newline 
It is not known if the sequence $\{\lambda_k^{0,1,p}\}_{k\in\mathbb{N}}$ exhausts all possible eigenvalues, except for the linear case $p=2$, where \eqref{auto} provides, cf. \cite[Theorem A.2]{BrascoPariniSquassina}, an exhaustive sequence for the usual eigenvalues of the Laplacian, 
\begin{equation*}
\lambda_{k}^{0,1,2}=\min\limits_{\mathcal{V}_k^{1,2}}\max\limits_{\mathcal{V}_k^{1,2}\cap\mathcal{S}^{1,2}(\Omega)}\int_{\Omega}|\nabla u|^2dx,
\end{equation*}
where 
\begin{equation*}
\mathcal{V}_k^{1,2}=\{A\subset W_0^{1,2}(\Omega):\ A\text{ vector space with }dim(A)\geq k\}.
\end{equation*}
In addition, the first eigenfunction $\varphi_1^{0,1,p}$ is unique and positive in $\Omega$, cf. \cite{Belloni2002,Vazquez1984}, while eigenfunctions associated to any other eigenvalue are sign-changing, cf. \cite{Kawohl2006}. We refer to \cite{Audoux2018} and references therein for other properties of symmetry of eigenfunctions and multiplicity of eigenvalues.\newline
On the other hand, the nonlocal nonlinear eigenvalue problem
\begin{equation}\label{eigenproblem_inf}
     \left\{\begin{array}{rl}
     (-\Delta_p)_{\infty}^su=\lambda|u|^{p-2}u &\quad\mbox{in}\quad \Omega,\\
                         u=0\mkern+59mu &\quad\mbox{on}\quad\!\Omega^c=\mathbb{R}^N\backslash\Omega,\\
		 \end{array}\right.
				\tag{$EP_{\infty}^{s,p}$}
\end{equation}
associated to the fractional $p\,$-Laplacian,
\begin{equation*}
(-\Delta_p)_{\infty}^su(x)=2P.V.\int_{\mathbb{R}^N}\frac{|u(x)-u(y)|^{p-2}(u(x)-u(y))}{|x-y|^{n+ps}}dy,
\end{equation*}
has spectral properties analogous to its local counterpart, cf. \cite{Brasco2016,Franzina,Iannizzotto2014,Lindgren2013}. Indeed, as it happens in the local case $s=1$, the eigenvalue problem \eqref{eigenproblem_inf} admits a sequence of increasing and positive eigenvalues 
\begin{equation*}
0<\lambda_{1}^{\infty,s,p}<\lambda_{2}^{\infty,s,p}\leq\ldots\leq\lambda_{k}^{\infty,s,p}\to+\infty\quad \text{as}\ k\to+\infty,
\end{equation*}
that can be characterized as 
\begin{equation*}
\lambda_{k}^{0,1,p}=\inf\limits_{A\in\mathcal{H}_k^{s,p}}\max\limits_{v\in A}\,[v]_{W^{s,p}(\mathbb{R}^N)}^p,
\end{equation*}
where $[\cdot]_{W^{s,p}(\mathbb{R}^N)}$ denotes the Gagliardo semi-norm
\begin{equation}\label{pregagli}
[v]_{W^{s,p}(\mathbb{R}^N)}^p\vcentcolon=\int_{\mathbb{R}^N}\int_{\mathbb{R}^N}\frac{|v(x)-v(y)|^p}{|x-y|^{N+ps}}dydx,
\end{equation}
and 
\begin{equation*}
\mathcal{H}_k^{s,p}\vcentcolon=\left\{A\subset\mathcal{S}^{s,p}(\Omega): A\text{ symmetric and compact, }i(A)\geq k\right\},
\end{equation*}
with
\begin{equation*}
\mathcal{S}^{s,p}(\Omega)\vcentcolon=\{u\in \mathcal{X}_0^{s,p}(\Omega): \|u\|_{L^p(\Omega)}=1\}
\end{equation*}
and
\begin{equation*}
\mathcal{X}_0^{s,p}(\Omega)\vcentcolon=\left\{u:\mathbb{R}^N\mapsto\mathbb{R}:[u]_{W^{s,p}(\mathbb{R}^N)}<+\infty,\ u=0\text{ in }\Omega^c\right\}.
\end{equation*}
As for the local case, the first eigenfunction $\varphi_1^{\infty,s,p}$ is unique and positive in $\Omega$, cf. \cite{Franzina}. Also, except for the linear case $p=2$, it is not known if the sequence $\{\lambda_{k}^{\infty,s,p}\}_{k\in\mathbb{N}}$ exhausts completely the spectrum of $(-\Delta_p)_{\infty}^s$.

Besides that, the continuity of the limit $s\to1^-$, for the eigenvalue problem \eqref{eigenproblem_inf} has been studied in \cite{BrascoPariniSquassina} where, by means of $\Gamma$-convergence arguments, it is proved that, for any $p\in(1,+\infty)$,
\begin{equation}\label{squa}
\lim\limits_{s\to1^-}(1-s)\lambda_{n}^{\infty,s,p}=K(N,p)\lambda_{n}^{0,1,p}\quad \text{for all}\ n\in\mathbb{N},
\end{equation}
with
\begin{equation}\label{squconst}
K(N,p)\vcentcolon=\frac{1}{p}\int_{\mathbb{S}^{N-1}}| e\cdot z |^pd\sigma(z),\quad e\in\mathbb{S}^{N-1},
\end{equation}
as well as the convergence of the corresponding normalized eigenfunctions in a suitable fractional norm. Observe that, due to symmetry reasons, the definition of $K(N,p)$ is indeed independent of the direction $e\in\mathbb{S}^{N-1}$. This result relies on an essential way on the fact that, under certain conditions, minimax values are continuous with respect to $\Gamma$-convergence,  cf. \cite{DegioMarco}. This nonlocal-to-local transition for nonlinear problems has been also addressed by means of $\Gamma$-convergence techniques in \cite{BonderSalort}, where stability as $s\to 1^-$ of solutions of fractional $p\,$-Laplacian equations with general nonlinear source terms is studied and, as a consequence, the stability as $s\to 1^-$ of the eigenvalues of the fractional $p\,$-Laplacian is derived.
\vspace{0.2cm}

Using the same techniques as for the nonlocal nonlinear eigenvalue problem \eqref{eigenproblem_inf}, a sequence of increasing and positive eigenvalues $\{\lambda_{k}^{\delta,s,p}\}_{k\in\mathbb{N}}$ of \eqref{eigenproblem} can be constructed. In this paper we show that, if we take $\delta\to0^+$, the variational eigenvalues $\{\lambda_{k}^{\delta,s,p}\}_{k\in\mathbb{N}}$ will converge, once appropriately scaled, to eigenvalues $\{\lambda_{k}^{0,1,p}\}_{k\in\mathbb{N}}$ of the local nonlinear eigenvalue problem \eqref{eigenproblem0}. 
We will also obtain that the associated eigenfunctions $\{\varphi_k^{\delta,s,p}\}_{k\in\mathbb{N}}$ of \eqref{eigenproblem} will converge, in the $L^p(\Omega)$-norm, to eigenfunctions $\{\varphi_k^{0,1,p}\}_{k\in\mathbb{N}}$ of \eqref{eigenproblem0}. When $\delta\to +\infty$, an analogous convergence result of eigenvalues and eigenfunctions of \eqref{eigenproblem} to those of \eqref{eigenproblem_inf} is also obtained. 

The linear case $p=2$ has been addressed by the authors, cf. \cite{bellido2020restricted}, showing the former spectral convergence behavior. In this linear setting the results are proved using $\Gamma$-convergence techniques together with the hilbertian structure of the associated functional space. In particular, the behavior of $(EP_{\delta}^{s,2})$ as $\delta \to 0^+$ is studied by combining a general $\Gamma$-convergence result for nonlocal nonlinear functionals, cf. \cite{BellCorPed}, together with the characterization of the eigenvalues and eigenfunctions based on projections onto orthogonal complements. The behavior as $\delta\to +\infty$ is an easy consequence of the monotonicity of the energy associated to $(-\Delta_p)^s_\delta$ with respect to $\delta>0$ but also strongly relies on the hilbertian structure of the problem. We refer to \cite{BellCorPed} and references therein for this subject in the linear framework. 

Since in the nonlinear case we have no such a Hilbert structure this approach is no longer valid. Nevertheless, we still use the $\Gamma$-convergence result of \cite{BellCorPed} that provides us with the limit energy functional of a suitable scaled version of the energy functional associated to \eqref{eigenproblem} as $\delta \to ^+$. Next, we prove the convergence of the minimax values by using a $\Gamma$-convergence result of \cite{DegioMarco} which guaranties the continuous dependence of the minimax values under the $\Gamma$-convergence of the functionals. Finally, we conclude the convergence of the eigenfunctions by following the approach of \cite{BrascoPariniSquassina}.

As in \cite{bellido2020restricted}, the obtained results lead us to claim that the peridynamic fractional $p\,$-Laplacian is an intermediate operator bridging the classical $p\,$-Laplacian $-\Delta_p$ and the fractional $p\,$-Laplacian $(-\Delta_p)_{\infty}^s$ for a fixed $s\in(0,1)$. 
\vspace{0.2cm} 

\textbf{Organization of the paper}: In Section \ref{functionalsetting} we introduce the appropriate functional setting to deal with the eigenvalue problem associated to $(-\Delta_p)_{\delta}^s$ and the precise statement of the main results of this work. In Section \ref{preliminaryresults} we collect the key results about $\Gamma$-convergence needed here. In Section \ref{horizon0} we prove Theorem \ref{theigen} which shows the convergence to the $p\,$-Laplacian spectrum as $\delta \to 0^+$. We finish in Section \ref{horizoninfty} with the proof Theorem \ref{theigen2} about the convergence to the fractional $p\,$-Laplacian spectrum as $\delta\to+\infty$. 
\section{Functional setting and Main Results}\label{functionalsetting}
The natural setting for equations involving the fractional $p\,$-Laplacian $(-\Delta_p)_{\infty}^s$ is the fractional Sobolev space $W^{s,p}(\mathbb{R}^N)$, defined as 
\begin{equation*}
W^{s,p}(\mathbb{R}^N)=\left\{v\in L^p(\mathbb{R}^N)\;:\; [v]_{W^{s,p}(\mathbb{R}^N)}<+\infty\right\}.
\end{equation*}
In order to deal with boundary value problems involving homogeneous Dirichlet boundary conditions $u=0$ in $\Omega^c=\mathbb{R}^N\backslash\Omega$, we consider the subspace
\begin{equation*}
\mathcal{X}_0^{\infty,s,p}(\Omega)\vcentcolon=\left\{u\in L^p(\Omega):[u]_{W^{s,p}(\mathbb{R}^N)}<+\infty,\ u=0\text{ in }\Omega^c\right\}.
\end{equation*}
Since $\Omega$ is a bounded domain with Lipschitz boundary, the space $\mathcal{X}_0^{\infty,s,p}(\Omega)$ coincides with the space defined as the completion of $C_0^{\infty}(\Omega)$ (smooth, compactly supported in $\Omega$ functions) respect to the semi-norm $[\cdot]_{W^{s,p}(\mathbb{R}^N)}$, cf. \cite[Prop. B.1]{BrascoPariniSquassina}. Observe that, since $u=0$ in $\Omega^c$, we have 
\begin{equation}\label{Gagliseminorm}
[v]_{W^{s,p}(\mathbb{R}^N)}^p=\iint\limits_{\mathcal{D}}\frac{|v(x)-v(y)|^p}{|x-y|^{N+ps}}dydx,
\end{equation}
where $\mathcal{D}=(\mathbb{R}^N\times\mathbb{R}^N)\backslash(\Omega^{c}\times\Omega^{c})$. 
Moreover, thanks to the fractional Sobolev inequality \cite{DiNezzaPalaValdi}, the semi-norm $[\cdot]_{W^{s,p}(\mathbb{R}^N)}^p$ is actually a norm on $\mathcal{X}_0^{\infty,s,p}(\Omega)$.
\vspace{0.2cm}

Next, given an horizon $\delta>0$, let us define the \textit{nonlocally} completed domain
\begin{equation*}
\Omega_{\delta}=\Omega\cup\partial_{\delta}\Omega=\{y\in\mathbb{R}^N:\ |x-y|<\delta,\ \text{for}\ x\in\Omega\},
\end{equation*}
and the functional space associated to $(-\Delta_p)_{\delta}^s$ as
\begin{equation*}
\mathcal{X}^{\delta,s,p}(\Omega_{\delta})\vcentcolon=\left\{u\in L^p(\Omega_{\delta}):[u]_{W^{\delta,s,p}(\Omega_{\delta})}<+\infty\right\}.
\end{equation*}
where $[\cdot]_{W^{\delta,s,p}(\Omega_{\delta})}$ denotes the semi-norm
\begin{equation*}
[u]_{W^{\delta,s,p}(\Omega_{\delta})}^p\vcentcolon=\int_{\Omega_{\delta}}\int_{\Omega_{\delta}\cap B(x,\delta)}\frac{|u(x)-u(y)|^p}{|x-y|^{N+sp}}dydx.
\end{equation*}
\begin{proposition}\label{belcor}\cite[Proposition 6.1]{BellCor} Let $s\in(0,1)$, $1\leq p<\infty$, $\delta>0$ and $\Omega\subset\mathbb{R}^N$ be a bounded Lipschitz domain. Then, there exists $C=C(\delta)>0$ such that for all $u\in W^{s,p}(\Omega)$,
\begin{equation*}
\int_{\Omega}\int_{\Omega}\frac{|u(x)-u(y)|^p}{|x-y|^{N+sp}}dydx\leq C\int_{\Omega}\int_{\Omega\cap B(x,\delta)}\frac{|u(x)-u(y)|^p}{|x-y|^{N+sp}}dydx.
\end{equation*}
\end{proposition}
Because of Proposition \ref{belcor}, the space $\mathcal{X}^{\delta,s,p}(\Omega_{\delta})$ is isomorphic to the (classical) fractional Sobolev space
\begin{equation*}
W^{s,p}(\Omega_{\delta})=\{u\in L^p(\Omega_{\delta}):\ [u]_{W^{s,p}(\Omega_{\delta})}<+\infty\},
\end{equation*}
where 
\begin{equation*}
[u]_{W^{s,p}(\Omega_{\delta})}^p=\int_{\Omega_{\delta}}\int_{\Omega_{\delta}}\frac{|u(x)-u(y)|^p}{|x-y|^{N+sp}}dydx.
\end{equation*}
To take into account the homogeneous Dirichlet boundary condition on $\partial_{\delta}\Omega$, we consider the subspace
\begin{equation*}
\mathcal{X}_0^{\delta,s,p}(\Omega)\vcentcolon=\left\{u\in \mathcal{X}^{\delta,s,p}(\Omega_\delta):\ u=0\text{ in }\partial_{\delta}\Omega\right\}.
\end{equation*}
\begin{lemma}\cite[Lemma 6.2]{BellCor}\label{poinc}
Let $\Omega\subset\mathbb{R}^N$ a bounded Lipschitz domain and let $\Omega_{D}$ a measurable set of $\Omega$ of positive measure. Let $s\in(0,1)$ and $1\leq p<\infty$. Then, there exists $C>0$ such that for all $u\in W^{s,p}(\Omega)$ with $u=0$ a.e. on $\Omega_{D}$ we have
\begin{equation*}
\|u\|_{L^p(\Omega)}\leq C\ |u|_{W^{s,p}(\Omega)}=C\left(\int_{\Omega}\int_{\Omega}\frac{|u(x)-u(y)|^p}{|x-y|^{N+sp}}dydx\right)^{\frac{1}{p}}.
\end{equation*}
\end{lemma}
By combining Lemma \ref{poinc} and Proposition \ref{belcor} it follows that the semi-norm $[\cdot]_{W^{\delta,s,p}(\Omega_{\delta})}$ is actually a norm on $\mathcal{X}_0^{\delta,s,p}(\Omega)$.\newline 
On the other hand, given a function $v\in \mathcal{X}_0^{\delta,s,p}(\Omega)$, although we have $v=0$ on $\partial_{\delta}\Omega=\Omega_{\delta}\backslash\Omega$, the norms $[v]_{W^{\delta,s,p}(\Omega_{\delta})}$ and
\begin{equation*}
[v]_{W^{\delta,s,p}(\Omega)}^p=\int_{\Omega}\int_{\Omega\cap B(x,\delta)}\frac{|v(x)-v(y)|^p}{|x-y|^{N+sp}}dydx.
\end{equation*}
are not the same. Indeed, if $v=0$ on $\partial_{\delta}\Omega$, denoting by 
\begin{equation*}
\begin{split}
\mathcal{D}_{\delta}\vcentcolon=& \left\{(x,y):\ x\in\Omega_{\delta}\wedge y\in\Omega_{\delta}\cap B(x,\delta)\quad \text{except}\quad x\in \partial_{\delta}\Omega \wedge y\in\partial_{\delta}\Omega\cap B(x,\delta)\right\}\\
=&\Big(\Omega_{\delta}\times\big(\Omega_{\delta}\cap B(x,\delta)\big)\Big)\Big\backslash\Big(\partial_{\delta}\Omega\times\big(\partial_{\delta}\Omega\cap B(x,\delta)\big)\Big),
\end{split}
\end{equation*}
we have  
\begin{equation}\label{normXcero}
[v]_{W^{\delta,s,p}(\Omega_{\delta})}^p=\iint\limits_{\mathcal{D}_{\delta}}\frac{|v(x)-v(y)|^p}{|x-y|^{N+ps}}dydx,
\end{equation}
and the strict inclusion $\big(\Omega\times(\Omega\cap B(x,\delta))\big)\subsetneq\mathcal{D}_{\delta}$. Hence, the norm $[v]_{W^{\delta,s,p}(\Omega_{\delta})}$ takes into account the interaction between $\Omega$ and $\partial_{\delta}\Omega$ in the sense that 
\begin{equation*}
\begin{split}
[v]_{W^{\delta,s,p}(\Omega_{\delta})}^p=&\iint\limits_{\mathcal{D}_{\delta}}\frac{|v(x)-v(y)|^p}{|x-y|^{N+2s}}dydx\\
=&\int_{\Omega}\int_{\Omega\cap B(x,\delta)}\frac{|v(x)-v(y)|^p}{|x-y|^{N+2s}}dydx\\
&+\int_{\partial_{\delta}\Omega}\int_{\Omega\cap B(x,\delta)}\frac{|v(y)|^p}{|x-y|^{N+2s}}dydx+\int_{\Omega}\int_{\partial_{\delta}\Omega\cap B(x,\delta)}\frac{|v(x)|^p}{|x-y|^{N+2s}}dydx.
\end{split}
\end{equation*}
Therefore, the space $\mathcal{X}_0^{\delta,s,p}(\Omega)$ is the appropriate space to deal with homogeneous elliptic boundary value problems involving the operator $(-\Delta_p)_{\delta}^s$.\newline 
Some comments are in order. Comparing the norms $[v]_{W^{s,p}(\mathbb{R}^N)}$ and $[v]_{W^{\delta,s,p}(\Omega_{\delta})}$ we observe that $\partial_{\delta}\Omega$ plays the role of $\Omega^c$. Indeed, the sets $\Omega_{\delta}$ and $\Omega_{\delta}\cap B(x,\delta)$ will lead to the complete space $\mathbb{R}^N$ in the limit $\delta\to+\infty$, the set $\Omega\cap B(x,\delta)$ will eventually reach the set $\Omega$ for $\delta>0$ big enough and the sets $\partial_{\delta}\Omega$ and $\partial_{\delta}\Omega\cap B(x,\delta)$ will reach $\Omega^c$ in the limit $\delta\to+\infty$. Moreover,
\begin{equation}\label{contenido}
\mathcal{D}_{\delta_1}\subset\mathcal{D}_{\delta_2}\quad\text{ for }\delta_1<\delta_2,
\end{equation}
and $\displaystyle \mathcal{D}_{\delta}\to\mathcal{D}$ as $\delta\to+\infty$ with $\mathcal{D}$ given in \eqref{Gagliseminorm}. Besides that, to study convergence phenomena when one takes $\delta\to+\infty$, it will be essential to study the relation between the usual fractional space $\mathcal{X}_0^{\infty,s,p}(\Omega)$ and the space $\mathcal{X}_0^{\delta,s,p}(\Omega)$. To that end, in Section \ref{horizoninfty} we will prove the following.
\begin{lemma}\label{isomorfia} For any $\delta>0$, the spaces $\mathcal{X}_0^{\infty,s,p}(\Omega)$ and $\mathcal{X}_0^{\delta,s,p}(\Omega)$ are isomorphic. In particular, there exists a constant $C=C(\delta)>1$ such that 
\begin{equation*}
[\cdot]_{W^{\delta,s,p}(\Omega_{\delta})}\leq[\cdot]_{W^{s,p}(\mathbb{R}^N)}\leq C(\delta) [\cdot]_{W^{\delta,s,p}(\Omega_{\delta})} \quad\text{for all }\delta>0.
\end{equation*}
Moreover, $C(\delta)\to 1$ as $\delta \to +\infty$. 
\end{lemma}
\vspace{0.2cm}
\begin{definition}
We say that $\lambda >0$ is an eigenvalue to problem \eqref{eigenproblem} if there exists $\varphi\in\mathcal{X}_0^{\delta,s,p}(\Omega)$ with $\varphi\ne 0$, such that 
\begin{equation*}
\int_{\Omega_{\delta}}\int_{\Omega_{\delta}\cap B(x,\delta)}\mkern-10mu\frac{|\varphi(x)-\varphi(y)|^{p-2}(\varphi(x)-\varphi(y))(v(x)-v(y))}{|x-y|^{N+ps}}dydx=\lambda\int_{\Omega}|\varphi|^{p-2}\varphi v\,dx,
\end{equation*}
for all $v\in\mathcal{X}_0^{\delta,s,p}(\Omega)$. In this case, we say that $\varphi$ is an eigenfunction of $(-\Delta_p)_{\delta}^s$ corresponding to the eigenvalue $\lambda>0$.
\end{definition}
Following the construction of the variational eigenvalues of \eqref{eigenproblem0} and \eqref{eigenproblem_inf}, the eigenvalues of the operator $(-\Delta_p)_{\delta}^s$ are defined as 
\begin{equation}\label{minimax}
\lambda_{k}^{\delta,s,p}(\Omega)\vcentcolon=\inf\limits_{A\in\mathcal{H}_k^{\delta,s,p}}\max\limits_{v\in A}\,[v]_{W^{\delta,s,p}(\Omega_{\delta})}^p,
\end{equation}
where
\begin{equation*}
\mathcal{H}_k^{\delta,s,p}\vcentcolon=\left\{A\subset\mathcal{S}^{\delta,s,p}(\Omega): A\text{ symmetric and compact, }i(A)\geq k\right\}
\end{equation*}
with
\begin{equation*}
\mathcal{S}^{\delta,s,p}(\Omega)\vcentcolon=\{u\in\mathcal{X}_0^{\delta,s,p}(\Omega): \|u\|_{L^p(\Omega)}=1\}.
\end{equation*}
Once we have introduced the appropriate functional setting, we present the main results of this work. 
\begin{theorem}\label{theigen} 
Let $\Omega\subset\mathbb{R}^N$ be and open bounded set with Lipschitz boundary. Then, for any $p\in(1,+\infty)$ and $k\in\mathbb{N}$, 
\begin{equation*}
\frac{p(1-s)}{\delta^{p(1-s)}}\lambda_{k}^{\delta,s,p}\to\gamma(N,p)\lambda_{k}^{0,1,p}\quad \text{as}\ \delta\to0^+
\end{equation*}
where 
\begin{equation}\label{constante}
\gamma(N,p)=\int_{\mathbb{S}^{N-1}}| e\cdot z |^pd\sigma(z),\quad e\in\mathbb{S}^{N-1}.
\end{equation}
Moreover, if $\varphi_{k}^{\delta,s,p}$ is an eigenfunction associated to the variational eigenvalue $\lambda_{k}^{\delta,s,p}$ with $\|\varphi_{k}^{\delta,s,p}\|_{L^p(\Omega)}=1$, there exists a subsequence $\{\varphi_{k}^{\delta_j,s,p}\}_{j\in\mathbb{N}}\subset\{\varphi_{k}^{\delta,s,p}\}_{\delta>0}$, with $\delta_j\to0$ as $j\to+\infty$, such that  
\begin{equation*}
\varphi_{k}^{\delta_j,s,p}\to\varphi_{k}^{0,1,p}\ \text{in}\ L^p(\Omega)\quad \text{as}\ j\to +\infty.
\end{equation*}
\end{theorem}
\begin{theorem}\label{theigen2}
Let $\Omega\subset\mathbb{R}^N$ be and open bounded set with Lipschitz boundary. Then, for any $p\in(1,+\infty)$, $s\in(0,1)$ and $k\in\mathbb{N}$, 
\begin{equation*}
\lambda_{k}^{\delta,s,p}\to\lambda_{k}^{\infty,s,p}\quad \text{as}\ \delta\to+\infty.
\end{equation*}
Moreover, if $\varphi_{k}^{\delta,s,p}$ is an eigenfunction associated to the variational eigenvalue $\lambda_{k}^{\delta,s,p}$ with $\|\varphi_{k}^{\delta,s,p}\|_{L^p(\Omega)}=1$, there exists a subsequence $\{\varphi_{k}^{\delta_j,s,p}\}_{j\in\mathbb{N}}\subset\{\varphi_{k}^{\delta,s,p}\}_{\delta>0}$, with $\delta_j\to+\infty$ as $j\to+\infty$, such that  
\begin{equation*}
\varphi_{k}^{\delta_j,s,p}\to\varphi_{k}^{\infty,s,p}\ \text{in}\ L^p(\Omega)\quad \text{as}\ j\to +\infty.
\end{equation*}
\end{theorem}
{It is remarkable that both the spectral asymptotic behavior described in \cite{BrascoPariniSquassina} for the fractional $p\,$-Laplacian as $s\to1^-$ (see \eqref{squa} and \eqref{squconst}) and the spectral asymptotic behavior of $(-\Delta_p)_{\delta}^s$ as $\delta\to0^+$ provided by Theorem \ref{theigen} coincide even so the fractionallity parameter $s\in(0,1)$ is fixed in our setting. Indeed, the simple relation $\gamma(N,p)=pK(N,p)$ lead us to 
\begin{equation*}
\lim\limits_{\delta\to0^+}\frac{p(1-s)}{\delta^{p(1-s)}}\lambda_{n}^{\delta,s,p}=\gamma(N,p)\lambda_{n}^{0,1,p}=\lim\limits_{s\to1^-}p(1-s)\lambda_{n}^{\infty,s,p},\quad\text{for all }n\in\mathbb{N}.
\end{equation*}
Roughly speaking, the scaling factor $\frac{p(1-s)}{\delta^{p(1-s)}}$ reflects both the blow up rate and the concentration rate of the nonlocal-to-local transitions described by the limits $s\to1^-$ and $\delta\to0^+$ respectively (see also Remark \ref{remarkBBM}).} 


\section{Preliminary results: $\Gamma$-convergence}\label{preliminaryresults}
The proofs of Theorem \ref{theigen} and Theorem \ref{theigen2} rely on some results about $\Gamma$-convergence of the norm $[\cdot]_{W^{\delta,s,p}(\Omega_{\delta})}$ with respect to $\delta$. $\Gamma$-convergence is the right concept of convergence for the study of limit processes of variational problems since, together with equicoercivity or compactness of sequences of uniformly bounded energy, it implies the convergence of infimum values and minimizers, if those exist. Furthermore, as showed in \cite{DegioMarco}, $\Gamma$-convergence behaves nicely with minimax problems. We will exploit this fact in the proofs of Theorem \ref{theigen} and Theorem \ref{theigen2}. A nice account on properties and applications of $\Gamma$-convergence can be found in \cite{Braides, Maso}. Next, we introduce the key result to deal with the limit $\delta\to0^+$.
\vspace{0.2cm}

Given $\delta>0$, let us consider the sequence of functionals $\{I_{\delta}^{s,p}\}_{\delta>0}$ defined as\break $I_{\delta}^{s,p}:L^p(\Omega)\mapsto[0,+\infty]$ with
\begin{equation}\label{fundelta}
I_{\delta}^{s,p}(u)\vcentcolon=\left\{ 
\begin{array}{ll} 
\frac{(p(1-s))^{\frac1p}}{\delta^{(1-s)}}[u]_{W^{\delta,s,p}(\Omega_{\delta})} &\mbox{if  }u\in\mathcal{X}_0^{\delta,s,p}(\Omega),\\
+\infty &\mbox{otherwise},
\end{array}\right.
\end{equation}
and $I_{0}^{1,p}:L^p(\Omega)\mapsto[0,+\infty]$ defined as
\begin{equation}\label{fun0}
I_{0}^{1,p}(u)\vcentcolon=\left\{ 
\begin{array}{ll} 
\gamma(N,p)^{\frac1p}\|\nabla u\|_{L^p(\Omega)} &\mbox{if  }u\in W_0^{1,p}(\Omega),\\
+\infty &\mbox{otherwise}, 
\end{array}\right.
\end{equation}
with $\gamma(N,p)$ defined in \eqref{constante}.\newline 
At a first stage we will prove that, as $\delta\to0^+$, the sequence of functionals $\{I_{\delta}^{s,p}\}_{\delta>0}$ $\Gamma$-converges, with respect to the strong-$L^p(\Omega)$ topology, to the functional $I_{0}^{1,p}$; namely, for every sequence $\{\delta_h\}_{h\in\mathbb{N}}$ of strictly decreasing positive numbers such that $\delta_h\to0$ as $h\to+\infty$, 
\begin{equation*}
\left(\Gamma-\lim\limits_{h\to+\infty}I_{\delta_h}^{s,p}\right)(u)=I_{0}^{1,p}(u)\quad\text{for all }u\in L^p(\Omega).
\end{equation*}
This will be simply denoted as
\begin{equation*}
\left(\Gamma-\lim\limits_{\delta\to0^+}I_{\delta}^{s,p}\right)(u)=I_{0}^{1,p}(u)\quad\text{for all }u\in L^p(\Omega).
\end{equation*}
To that end, we consider a general functional of the form 
\begin{equation*}
I(u)=\int_{\Omega}\int_{\Omega\cap B(x,\delta)}\omega(x-y,u(x)-u(y))dydx,
\end{equation*}
for a \textit{potential function} $\omega(x,y):\mathbb{R}^N\times\mathbb{R}\mapsto\mathbb{R}$ verifying that, for some $\beta\in\mathbb{R}$, the following limit exists,
\begin{equation}\label{blowup}
\omega^{\circ}(x,y)=\lim\limits_{t\to0^+}\frac{1}{t^{\beta}}\omega(tx,ty).
\end{equation}
Let $\overline{\omega}^{c}:\mathbb{R}^N\mapsto\mathbb{R}$ be the \textit{limit density convexification} of $\overline\omega$,
\begin{equation*}
\overline{\omega}^{c}=\sup\{v:\ v\leq \overline{\omega}\ \text{and}\ v\ \text{convex}\},
\end{equation*}
where $\overline{\omega}:\mathbb{R}^N\mapsto\mathbb{R}$ is the \textit{limit density} of $\omega$,
\begin{equation*}
\overline{\omega}(F)=\int_{\mathbb{S}^{N-1}}\omega^{\circ}(z,Fz)d\sigma(z).
\end{equation*}
Under the hypotheses stated below, given the sequence of rescaled functionals
\begin{equation}\label{scaledfunctional}
I_{\delta}(u)\vcentcolon=\frac{N+\beta}{\delta^{N+\beta}}\int_{\Omega}\int_{\Omega\cap B(x,\delta)}\omega(x-y,u(x)-u(y))dydx.
\end{equation}
we have 
\begin{equation}\label{gl}
\left(\Gamma-\lim\limits_{\delta\to0^+}I_{\delta}\right)(u)=I_{0}(u)\quad\text{for all }u\in L^p(\Omega),
\end{equation}
with
\begin{equation*}
I_0(u)\vcentcolon=\int_{\Omega}\overline{\omega}^{c}(\nabla u)dx.
\end{equation*}
In particular, the $\Gamma$-convergence \eqref{gl} is ensured by the next result that also provides the compactness of uniformly bounded energy sequences.\newline 
Let us set $\tilde{\Omega}\vcentcolon=\{ z=x-y\;:\; x,\,y\in \Omega\}$ and $\mathcal{A}_{\delta}\vcentcolon=\{v\in L^p(\Omega): v=0\ \text{on}\ \partial_{\delta}\Omega\}$.

\begin{theorem}\cite[Theorem 1]{BellCorPed}\label{belcorped} Let $\Omega\subset\mathbb{R}^N$ be a bounded domain with Lipschitz boundary and $\omega:\tilde{\Omega}\times\mathbb{R}\mapsto\mathbb{R}$ satisfying the hypotheses (H1)-(H5) below. Then, the following holds:
\begin{itemize}
\item[a)]{\it Compactness:} For each $\delta>0$, let $u_{\delta}\in \mathcal{A}_{\delta}$ such that
\begin{equation*}
\sup\limits_{\delta} I_{\delta}(u_{\delta})<+\infty.
\end{equation*}
Then, there exist $u\in W_0^{1,p}(\Omega)$ such that, for a subsequence (that we do not relabel),
\begin{equation*}
u_{\delta}\to u\ \text{strong in}\ L^p(\Omega)\qquad\text{as } \delta\to0^+.
\end{equation*}
\item[b)] $\Gamma$-liminf inequality: For each $\delta>0$ let $u_{\delta}\in\mathcal{A}_{\delta}$ and $u\in W_0^{1,p}(\Omega)$ such that $u_{\delta}\to u$ strong in $L^p(\Omega)$ as $\delta\to0^+$. Then,
\begin{equation*}
I_0(u)\leq \liminf\limits_{\delta\to0^+}I_{\delta}(u_{\delta}).
\end{equation*}
\item[c)] $\Gamma$-limsup inequality: For each $\delta>0$ and $u\in W_0^{1,p}(\Omega)$ there exist $u_{\delta}\in\mathcal{A}_{\delta}$, called \textit{recovery sequence}, such that $u_{\delta}\to u$ strong in $L^p(\Omega)$ as $\delta\to0^+$ and 
\begin{equation*}
\limsup\limits_{\delta\to0^+}I_{\delta}(u_{\delta})\leq I_0(u).
\end{equation*}
\end{itemize}
\end{theorem}
For a general potential function $\omega(x,y)$, the hypotheses required in Theorem \ref{belcorped} are quite involved but, as it is noted in \cite{BellCorPed}, for a potential function of the form
\begin{equation*}
\omega(x,y)=f(x)g(y),
\end{equation*} 
the necessary hypotheses are the following:
\begin{itemize}
\item[H1)] $f$ is Lebesgue measurable and $g$ is Borel measurable and convex.
\item[H2)] There exists constants $c_0,c_1>0$ and $h\in L^1(\mathbb{S}^{N-1})$ with $h\geq0$ such that, for some $1<p<+\infty$ and $0\leq\alpha<N+p$,
\begin{equation*}
c_0\frac{|y|^p}{|x|^{\alpha}}\leq f(x)g(y)\leq c_1 h\left(\frac{x}{|x|}\right)\frac{|y|^p}{|x|^{\alpha}}\quad\text{ for }x\in\widetilde{\Omega},\ y\in\mathbb{R}.
\end{equation*}
\item[H3)] The functions $f^{\circ}:\mathbb{R}\backslash\{0\}\mapsto\mathbb{R}$ and $g^{\circ}:\mathbb{R}\mapsto\mathbb{R}$ defined as
\begin{equation*}
f^{\circ}(x)\vcentcolon=\lim\limits_{t\to0^+}t^{\alpha}f(tx)\quad\text{and}\quad g^{\circ}(y)\vcentcolon=\lim\limits_{t\to0^+}\frac{1}{t^p}g(ty),
\end{equation*}
are continuous and, for each compact $K\subset\mathbb{R}$,
\begin{equation*}
\lim\limits_{t\to0^+}\sup\limits_{x\in\mathbb{S}^{N-1}}|t^{\alpha}f(tx)-f^{\circ}(x)|=0\quad\text{and}\quad \lim\limits_{t\to0^+}\sup\limits_{K\subset\mathbb{R}}|\frac{1}{t^p}g(ty)-g^{\circ}(y)|=0.
\end{equation*}
\end{itemize}
\vspace{0.2cm}
A straightforward consequence, cf. \cite{Braides}, of the $\Gamma$-convergence and the compactness provided by Theorem \ref{belcorped} is the convergence of minimizers for $I_\delta$, whose existence is guaranteed under the previous hypotheses, cf. \cite{BellCor}. 
\begin{corollary}\label{corGammaconv} In the conditions of Theorem \ref{belcorped}, for each $\delta>0$, let $u_\delta\in\mathcal{X}_0^{\delta,s,p}(\Omega)$ be a minimizer of $I_\delta$. Then, there exists $u_0\in W_0^{1,p}(\Omega)$ a minimizer of $I_0$ such that, up to a subsequence,
\begin{equation*}
u_\delta\to u_0\mbox{  strong in  }L^p(\Omega)\quad\text{as }\delta\to 0^+,
\end{equation*}
and 
\begin{equation*}
\lim\limits_{\delta\to0^+}I_{\delta}(u_{\delta})=I_0(u_0).
\end{equation*} 
\end{corollary} 
As we see, the $\Gamma$-convergence for the limit $\delta\to0^+$ is a delicate issue that involves a precise rescaling and the construction of the $\Gamma$-limit functional. Instead, the $\Gamma$-convergence in the case $\delta\to+\infty$ will be an easy consequence of monotony of $[\cdot]_{W^{\delta,s,p}(\Omega_{\delta})}$ with respect to $\delta$.\newline
On the other hand, although Theorem \ref{belcorped} provides us with the $\Gamma$-convergence of the norm $[\cdot]_{W^{\delta,s,p}(\Omega_{\delta})}^p$ as $\delta\to0^+$, we also need to guarantee the convergence of the minimax energies \eqref{minimax}. This step is accomplished using some results about the behavior of the minimax values under $\Gamma$-convergence, cf. \cite{DegioMarco}, that we present next for the shake of completeness.\newline 
Let $X$ be a metrizable and locally convex topological vector space and, for every $h\in \mathbb{N}$, let $f_h:X\mapsto [0,+\infty]$ and $g_h:X\backslash\{0\}\mapsto\mathbb{R}$ be two functions such that:
\begin{itemize}
\item[1)] $f_h$ and $g_h$ are both even and positively homogeneous of degree 1,
\item[2)] $f_h$ is convex,
\item[3)] for every $b\in\mathbb{R}$, the restriction of $g_h$ to $\{u\in X\backslash\{0\}: f_h(u)\leq b\}$ is continuous.
\end{itemize}
Let us denote by $\mathcal{C}$ the family of nonempty compact subsets of $X$ and 
\begin{equation*}
\mathcal{C}_k^{h}\vcentcolon=\big\{A\subset\{u\in X: g_h(u)=1\}: A\text{ symmetric and compact, }i(A)\geq k\big\}
\end{equation*}
and define $\mathcal{F}_h^{(k)}:\mathcal{C}\mapsto[0,+\infty]$ as
\begin{equation*}
\mathcal{F}_h^{(k)}(A)=\left\{ 
\begin{array}{ll} 
\max\limits_{u\in A} f_h(u)&\mbox{if  }A\in\mathcal{C}_k^{h},\\
+\infty &\mbox{otherwise}.
\end{array}\right.
\end{equation*}
The set $\mathcal{C}$ will be endowed with the topology induced by the Hausdorff distance inherited from the metric of $X$. Finally, let $f:X\mapsto[0,+\infty]$ and $g:X\mapsto\mathbb{R}$ be two even functions such that $g(0)=0$ and define $\mathcal{C}_k\subseteq\mathcal{C}$ and $\mathcal{F}^{(k)}:\mathcal{C}\mapsto[0,+\infty]$ in an analogous way. The result dealing with the convergence of the minimax energies reads as follows.

\begin{proposition}\cite[Corollary 4.4]{DegioMarco}\label{degiomarzo}
Assume that 
\begin{equation*}
f(u)=\left(\Gamma-\lim\limits_{h\to+\infty}f_h\right)(u),\quad\text{for all } u\in X,
\end{equation*}
and that, for every strictly increasing sequence $\{h_n\}\subset\mathbb{N}$ and every sequence $\{u_n\}_{n\in\mathbb{N}}\subset X\backslash\{0\}$ such that 
\begin{equation*}
\sup\limits_{n\in\mathbb{N}}f_{h_n}(u_n)<+\infty,
\end{equation*}
there exists a subsequence $\{u_{n_j}\}_{j\in\mathbb{N}}\subset X\backslash\{0\}$ converging to some $u\in X$ with 
\begin{equation*}
\lim\limits_{j\to+\infty}g_{h_{n_j}}(u_{n_j})=g(u).
\end{equation*}
Then, for every integer $k\geq1$, the sequence $\{\mathcal{F}_h^{(k)}\}_{h\in\mathbb{N}}$ is asymptotically coercive and
\begin{equation*}
\mathcal{F}^{(k)}(A)=\left(\Gamma-\lim\limits_{h\to+\infty}\mathcal{F}_h^{(k)}\right)(A)\quad\text{for all }A\in\mathcal{C},
\end{equation*}
\begin{equation*}
\inf\limits_{A\in\mathcal{C}}\mathcal{F}^{(k)}(A)=\lim\limits_{h\to+\infty}\left(\inf_{A\in\mathcal{C}}\mathcal{F}_h^{(k)}(A)\right),
\end{equation*}
\begin{equation*}
\inf\limits_{A\in\mathcal{C}_k}\sup\limits_{A}f=\lim\limits_{h\to+\infty}\left(\inf\limits_{A\in\mathcal{C}_k^{h}}\sup\limits_{A}f_h\right).
\end{equation*}
\end{proposition}

\begin{proposition}\cite[Corollary 3.3]{DegioMarco}\label{degiomarzo1}
Let $Y$ a vector subspace of $X$ such that 
\begin{equation*}
\{u\in X\backslash\{0\}:\ g(u)>0\text{ and }f(u)<+\infty\}\subseteq Y
\end{equation*}
and let $\tau_Y$ be any topology on $Y$ which makes $Y$ a metrizable and locally convex topological vector space such that, for every $b\in\mathbb{R}$, the restriction of $g$ to
\begin{equation*}
\{u\in Y\backslash\{0\}:\ f(u)\leq b\}
\end{equation*}
is $\tau_Y$-continuous. Then the minimax values defined in the space $Y$ agree with those defined in the
originary space $X$.
\end{proposition}
In particular, Proposition \ref{degiomarzo} allows us to take the limit on $\delta$ on the minimax values defined \eqref{minimax} while Proposition \ref{degiomarzo1} ensures that the minimax values are not affected by the change of topology produced by restricting the energy functionals.
\vspace{0.2cm}
 
\section{Taking the horizon $\delta\to0^+$}\label{horizon0}
This section is devoted to the proof of Theorem \ref{theigen}. We start by proving a preliminary result concerning the $\Gamma$-limit of the norm $[\cdot]_{W^{\delta,s,p}(\Omega_{\delta})}$.
\begin{lemma}\label{Glimit} Let us consider the rescaled functional 
\begin{equation}\label{funcionaldeltas}
\mathcal{E}_{\delta}^{s,p}(u)\vcentcolon=\frac{p(1-s)}{\delta^{p(1-s)}}[u]_{W^{\delta,s,p}(\Omega_{\delta})}^p
\end{equation} 
defined on $\mathcal{X}_0^{\delta,s,p}(\Omega)$. Then, the $\Gamma$-limit of $\mathcal{E}_{\delta}^{s,p}$ as $\delta\to 0^+$ is given by
\begin{equation}\label{GlimitFunc}
\mathcal{E}_{0}^{1,p}(u)=\gamma(N,p)\int_{\Omega}|\nabla u|^pdx,
\end{equation}
with $\gamma(N,p)$ defined in \eqref{constante}. 
\end{lemma}

\begin{proof}
Since the norm $[\cdot]_{W^{\delta,s,p}(\Omega_{\delta})}$ involves the potential function 
\begin{equation*}
\omega_{s,p}(x,y)=\frac{|y|^p}{|x|^{N+ps}},
\end{equation*}
we have $\alpha=N+ps<N+p$ and hypotheses H1)-H3) are clearly satisfied. Next, let us calculate the \textit{limit density convexification} of $\omega_{s,p}$. As the potential function $\omega_{s,p}$ is homogeneous,  condition \eqref{blowup} is satisfied with $\beta=p-(N+ps)$, so that the appropriate scaling in \eqref{scaledfunctional} is $\frac{p(1-s)}{\delta^{p(1-s)}}$, and
\begin{equation*}
\omega^{\circ}_{s,p}(x,y)=\lim\limits_{t\to0^+}\frac{1}{t^\beta}\omega_{s,p}(tx,ty)=\lim\limits_{t\to0^+}\frac{1}{t^{p-(N+ps)}}\omega_{s,p}(tx,ty)=\omega_{s,p}(x,y).
\end{equation*}
Then, since $\omega_{s,p}(x,y)=\omega_{s,p}(|x|,|y|)$, the limit density function is given by
\begin{equation*}
\begin{split}
\overline{\omega}_{s,p}(F)&=\int_{\mathbb{S}^{N-1}}\omega_{s,p}^{\circ}(z,F\cdot z)d\sigma(z)=\int_{\mathbb{S}^{N-1}}\omega_{s,p}(z,F\cdot z)d\sigma(z)\\
&=\int_{\mathbb{S}^{N-1}}\omega_{s,p}(|z|,|F\cdot z|)d\sigma(z)=\int_{\mathbb{S}^{N-1}}\omega_{s,p}(1,|F\cdot z|)d\sigma(z)\\
&=\int_{\mathbb{S}^{N-1}}|F\cdot z|^pd\sigma(z)=\gamma(N,p)|F|^p,
\end{split}
\end{equation*}
where $\displaystyle \gamma(N,p)\vcentcolon=\int_{\mathbb{S}^{N-1}}|e\cdot z|^pd\sigma(z)$ independent of $e=\frac{F}{|F|}\in\mathbb{S}^{N-1}$.\newline 
Let us observe that the denominator of $\omega_{s,p}^{\circ}(x,y)$ plays no role in the integration over the unit sphere $\mathbb{S}^{N-1}$, so that no dependence on $s$ is conserved for the limit density $\overline{\omega}_{s,p}(F)=\overline{\omega}_{p}(F)=\gamma(N,p) |F|^p$. At last, since $p>1$, the limit density $\overline{\omega}_{p}(F)$ is convex and, then, the convexification $\overline{\omega}_{p}^{c}(F)=\overline{\omega}_{p}(F)$. Hence, setting
\begin{equation*}
\mathcal{E}_{0}^{1,p}(u)\vcentcolon=\int_{\Omega}\overline{\omega}_{p}^{c}(\nabla u)dx=\gamma(N,p) \int_{\Omega}|\nabla u|^pdx,
\end{equation*}
because of Theorem \ref{belcorped}, we have: 
\begin{enumerate}
\item  \textit{$\Gamma$-liminf inequality}: Given $u_{\delta}\in\mathcal{A}_{\delta}$ and $u\in W_0^{1,p}(\Omega)$ such that $u_\delta\to u$ strong in $L^p(\Omega)$, we have 
\begin{equation*}
\mathcal{E}_{0}^{1,p}(u)\leq\liminf_{\delta\to 0^+} \mathcal{E}_{\delta}^{s,p}(u_\delta).
\end{equation*} 

\item \textit{$\Gamma$-limsup inequality}: Given $u\in W_0^{1,p}(\Omega)$, there exists a recovery sequence $u_\delta\in \mathcal{A}_{\delta}$ such that $u_\delta\to u$ strong in $L^p(\Omega)$ as $\delta \to 0^+$ and
\begin{equation*}
\limsup_{\delta\to 0^+}\mathcal{E}_{\delta}^{s,p}(u_\delta) \leq \mathcal{E}_{0}^{1,p}(u).
\end{equation*}
\end{enumerate}
We conclude,
\begin{equation*}
\left(\Gamma-\lim\limits_{\delta\to0^+}\mathcal{E}_{\delta}^{s,p}\right)(u)=\mathcal{E}_{0}^{1,p}(u)\quad\text{for all }u\in L^p(\Omega).
\end{equation*}
\end{proof}
\begin{proof}[Proof of Theorem \ref{theigen}]\hfill\newline
The proof is done in several steps. First we prove the $\Gamma$-convergence of the functional $I_{\delta}^{s,p}$ to the functional $I_{0}^{1,p}$ as $\delta\to0^+$. This is derived easily from Lemma \ref{Glimit}. To continue, we prove the convergence of the minimax values given in \eqref{minimax} using the compactness provided by Theorem \ref{belcorped}-a), Propositions \ref{degiomarzo} and \ref{degiomarzo1}. At last we prove the convergence of the associated eigenfunctions as in \cite[Theorem 1.2]{BrascoPariniSquassina}.\newline 
Let us write 
\begin{equation*}
I_{\delta}^{s,p}(u)=\left\{ 
\begin{array}{ll} 
(\mathcal{E}_{\delta}^{s,p}(u))^{\frac1p} &\mbox{if  }u\in\mathcal{X}_0^{\delta,s,p}(\Omega),\\
+\infty &\mbox{otherwise},
\end{array}\right.
\end{equation*}
and
\begin{equation*}
I_{0}^{1,p}(u)=\left\{ 
\begin{array}{ll} 
(\mathcal{E}_{0}^{1,p}(u))^{\frac1p} &\mbox{if  }u\in W_0^{1,p}(\Omega),\\
+\infty &\mbox{otherwise}, 
\end{array}\right.
\end{equation*}
for the functionals $\mathcal{E}_{\delta}^{s,p}$ and $\mathcal{E}_{0}^{1,p}$ given by \eqref{funcionaldeltas} and \eqref{GlimitFunc} respectively. Then, we have
\begin{enumerate}
\item  \textit{$\Gamma$-liminf inequality}: Given $u_{\delta}\in\mathcal{A}_{\delta}$ and $u\in W_0^{1,p}(\Omega)$ such that $u_\delta\to u$ strong in $L^p(\Omega)$, we can assume that 
\begin{equation*}
\liminf_{\delta\to 0^+}\mathcal{E}_{\delta}^{s,p}(u_\delta)<+\infty,
\end{equation*}
since otherwise there is nothing to prove. Hence, for any $\delta>0$ small enough we have $u_{\delta}\in \mathcal{X}_0^{\delta,s,p}(\Omega)$. On the other hand, because of Lemma \ref{Glimit}, we have 
\begin{equation*}
\mathcal{E}_{0}^{1,p}(u)\leq\liminf_{\delta\to 0^+} \mathcal{E}_{\delta}^{s,p}(u_\delta).
\end{equation*} 
Therefore, we get
\begin{equation*}
I_{0}^{1,p}(u)\leq\liminf_{\delta\to 0^+} I_{\delta}^{s,p}(u_\delta).
\end{equation*}
\item \textit{$\Gamma$-limsup inequality}: Given $u\in W_0^{1,p}(\Omega)$, because of Lemma \ref{Glimit}, there exists a recovery sequence $u_\delta\in \mathcal{A}_{\delta}$ such that $u_\delta\to u$ strong in $L^p(\Omega)$ as $\delta \to 0^+$ and
\begin{equation*}
\limsup_{\delta\to 0^+}\mathcal{E}_{\delta}^{s,p}(u_\delta) \leq \mathcal{E}_{0}^{1,p}(u),
\end{equation*}
so that
\begin{equation*}
\limsup_{\delta\to 0^+}I_{\delta}^{s,p}(u_\delta) \leq I_{0}^{1,p}(u).
\end{equation*}
\end{enumerate}
Then, we conclude
\begin{equation}\label{h.1}
\left(\Gamma-\lim\limits_{h\to+\infty}I_{\delta_h,s,p}\right)(u)=I_{0}^{1,p}(u)\quad\text{for all }u\in L^p(\Omega).
\end{equation}
for every sequence $\{\delta_h\}_{h\in\mathbb{N}}$ of strictly decreasing positive numbers such that $\delta_h\to0$ as $h\to+\infty$.\newline
Let us note that, since the first eigenvalues of \eqref{eigenproblem} and \eqref{eigenproblem0} are global minimums, i.e.,
\begin{equation*}
\lambda_{1}^{\delta,s,p}=\min\limits_{u\in\mathcal{S}^{\delta,s,p}(\Omega)}[u]_{W^{\delta,s,p}(\Omega_{\delta})}^p\quad\text{   and   }\quad \lambda_{1}^{0,1,p}=\min\limits_{u\in\mathcal{S}^{1,p}(\Omega)}\|\nabla u\|_{L^p(\Omega)}^p,
\end{equation*}
because of Corollary \ref{corGammaconv} we deduce that
\begin{equation*}
\frac{p(1-s)}{\delta^{p(1-s)}}\lambda_{1}^{\delta,s,p}\to\gamma(N,p)\lambda_{1}^{0,1,p}\quad \text{as}\ \delta\to0^+.
\end{equation*}
To continue, let $1<p<\infty$ and consider the functional $g_p: L^p(\Omega)\mapsto [0,+\infty)$ defined as
\begin{equation*}
g_p(u)\vcentcolon=\|u\|_{L^p(\Omega)}.
\end{equation*}
Next, by Theorem \ref{belcorped}-a), for every strictly increasing sequence $\{h_n\}_{n\in\mathbb{N}}\subset\mathbb{N}$ such that $\delta_{h_n}\to0$ as $h_n\to+\infty$ and any sequence $\{u_{\delta_{h_n}}\}_{n\in\mathbb{N}}$ with $u_{\delta_{h_n}}\in\mathcal{A}_{\delta_{h_n}}$ and
\begin{equation}\label{h.2}
\sup\limits_{n\in\mathbb{N}}I_{\delta_{h_n}}^{s,p}(u_{\delta_{h_n}})<+\infty,
\end{equation} 
there exists $u\in W_0^{1,p}(\Omega)$ and a subsequence simply denoted as $\{u_{\delta_j}\}_{j\in\mathbb{N}}$, such that 
\begin{equation*}
u_{\delta_j}\to u \text{ strong in }L^p(\Omega)\quad \text{as }j\to+\infty.
\end{equation*}
Hence, by definition,
\begin{equation}\label{h.3}
\lim\limits_{j\to+\infty}g_p(u_{\delta_j})=g_p(u).
\end{equation}
Because of \eqref{h.1}, \eqref{h.2} and \eqref{h.3}, the functionals $I_{\delta}^{s,p}$ and $g_p$ satisfy the hypotheses of Proposition \ref{degiomarzo} and, therefore,
\begin{equation}\label{preconv}
\lim\limits_{h\to+\infty}\left(\inf\limits_{A\in\mathcal{C}_{k}^{p}}\sup\limits_{u\in A}I_{\delta_h}^{s,p}(u)\right)=\inf\limits_{A\in\mathcal{C}_{k}^{p}}\sup\limits_{u\in A}I_{0}^{1,p}(u),
\end{equation}
where 
\begin{equation*}
\mathcal{C}_k^{p}\vcentcolon=\big\{A\subset\{u\in L^p(\Omega): g_p(u)=1\}: A\text{ symmetric and compact, }i(A)\geq k\big\}.
\end{equation*}
Moreover, again by Theorem \ref{belcorped}-a), for every $b\in\mathbb{R}$ the restriction of $g_p$ to\break $\{u\in L^p(\Omega): I_{\delta}^{s,p}(u)\leq b \}$ is continuous and therefore, by Proposition \ref{degiomarzo1}, we find that
\begin{equation*}
\inf\limits_{A\in\mathcal{C}_{k}^{p}}\sup\limits_{u\in A}I_{0}^{1,p}(u)=\inf\limits_{A\in\mathcal{H}_{k}^{1,p}}\sup\limits_{u\in A}I_{0}^{1,p}(u)
\end{equation*}
and
\begin{equation*}
\inf\limits_{A\in\mathcal{C}_{k}^{p}}\sup\limits_{u\in A}I_{\delta_n,s,p}(u)=\inf\limits_{A\in\mathcal{H}_{k}^{\delta_n,s,p}}\sup\limits_{u\in A}I_{\delta_n,s,p}(u).
\end{equation*}
Hence, the minimax values with respect to the $\mathcal{X}_0^{\delta,s,p}(\Omega)$-topology are equal to those with respect to the weaker topology $L^p(\Omega)$. Thus, for any sequence $\{\delta_n\}_{n\in\mathbb{N}}$ of positive numbers such that $\delta_n\to0$ as $n\to+\infty$, we conclude
\begin{equation}\label{ceig}
\frac{p(1-s)}{\delta_n^{p(1-s)}}\lambda_{k}^{\delta_n,s,p}\to\gamma(N,p)\lambda_{k}^{0,1,p}\quad \text{as}\ n\to+\infty.
\end{equation}
Finally, it remains to prove the assertion about the convergence of the eigenfunctions $\{\varphi_k^{\delta,s,p}\}_{k\in\mathbb{N}}$ as $\delta\to0^+$.\newline 
Fixed $k\in\mathbb{N}$, we observe that, if $\varphi_k^{\delta,s,p}\in\mathcal{X}_0^{\delta,s,p}(\Omega)$ is an eigenfunction associated to the eigenvalue $\lambda_{k}^{\delta,s,p}$ and such that $\|\varphi_k^{\delta,s,p}\|_{L^p(\Omega)}=1$, then 
\begin{equation*}
[\varphi_k^{\delta,s,p}]_{W^{\delta,s,p}(\Omega_{\delta})}^p=\lambda_{k}^{\delta,s,p}.
\end{equation*}
Hence, because of the eigenvalue convergence proved before, for $\delta>0$ small enough, we have
\begin{equation}\label{bound1}
I_{\delta}^{s,p}(\varphi_k^{\delta,s,p})\leq\left(\gamma(N,p)\lambda_{k}^{0,1,p}\right)^{\frac1p}+C<+\infty
\end{equation}
for some constant $C>0$. Then, because of Theorem \ref{belcorped}-a), there exists $\psi\in W_0^{1,p}(\Omega)$ with $\|\psi\|_{L^p(\Omega)}=1$ and a sequence $\{\delta_n\}_{n\in\mathbb{N}}$ of positive numbers with $\delta_n\to0$ as $n\to+\infty$ such that 
\begin{equation}\label{ceigf}
\varphi_k^{\delta_n,s,p}\to\psi\text{ strong in }L^p(\Omega)\quad\text{as }n\to+\infty.
\end{equation}
On the other hand, each $\varphi_k^{\delta_n,s,p}$ satisfies
\begin{equation*}
     \left\{\begin{array}{rl}
     (-\Delta_p)_{\delta_n}^s\varphi_k^{\delta_n,s,p}=\lambda_k^{\delta_n,s,p}|\varphi_k^{\delta_n,s,p}|^{p-2}\varphi_k^{\delta_n,s,p} &\quad\mbox{in}\quad \Omega,\\
                         u=0\mkern+178mu &\quad\mbox{on}\quad \partial_{\delta_n}\Omega,\\
		 \end{array}\right.
\end{equation*} 
so that it is the unique minimizer of the problem
\begin{equation*}
\min\limits_{u\in L^p(\Omega)}\mathcal{F}_{\delta_n}(u)\vcentcolon=\min\limits_{u\in L^p(\Omega)}\left\{ \mathcal{E}_{\delta_n,s,p}(u)+p\int_{\Omega}F_{\delta_n}udx\right\},
\end{equation*}
with $\mathcal{E}_{\delta_n,s,p}$ defined in \eqref{funcionaldeltas} and $F_{\delta_n}=-\frac{p(1-s)}{\delta_n^{p(1-s)}}\lambda_k^{\delta_n,s,p}|\varphi_k^{\delta_n,s,p}|^{p-2}\varphi_k^{\delta_n,s,p}\in L^{p'}(\Omega)$. Furthermore, due to \eqref{ceig} and \eqref{ceigf}, we have $F_{\delta_n}\to F_0$ strong in $L^{p'}(\Omega)$ with 
\begin{equation*}
F_0=-\gamma(N,p)\lambda_{k}^{0,1,p}|\psi|^{p-2}\psi.
\end{equation*}
Hence, by \cite[Proposition 6.25]{Maso}, we have
\begin{equation*}
\left(\Gamma-\lim\limits_{n\to+\infty}\mathcal{F}_{\delta_n}\right)(u)=\mathcal{F}_{0}(u)\quad\text{for all }u\in L^p(\Omega),
\end{equation*}
with 
\begin{equation*}
\mathcal{F}_{0}(u)=\mathcal{E}_{0}^{1,p}(u)+p\int_{\Omega}F_{0}udx.
\end{equation*}
Because of Corollary \ref{corGammaconv} and the strict convexity we get that $\psi\in W_0^{1,p}(\Omega)$ is the unique minimizer of the limit problem
\begin{equation*}
\min\limits_{u\in L^p(\Omega)}\mathcal{F}_{0}(u)\vcentcolon=\min\limits_{u\in L^p(\Omega)}\left\{ \mathcal{E}_{0}^{1,p}(u)+p\int_{\Omega}F_{0}udx\right\}.
\end{equation*}
Thus, $\psi\in W_0^{1,p}(\Omega)$ is a (weak) solution of the associated Euler-Lagrange equation, i.e., 
\begin{equation*}
     \left\{\begin{array}{rl}
     -\Delta_p\psi=\lambda_k^{0,1,p}|\psi|^{p-2}\psi &\quad\mbox{in}\quad \Omega,\\
                         \psi=0\mkern+94.45mu &\quad\mbox{on}\quad \partial\Omega,\\
		 \end{array}\right.
\end{equation*} 
and we conclude that $\psi$ is a normalized eigenfunction associated to the eigenvalue $\lambda_k^{0,1,p}$, so that $\varphi_k^{\delta_n,s,p}\to \varphi_k^{0,1,p}$ strong in $L^p(\Omega)$.
\end{proof}

\begin{remark} \label{remarkBBM}
The \textit{$\Gamma$-limsup inequality} proved above in the proof of Theorem \ref{theigen}, can be also obtained using the classical {\it localization} result of Bourgain, Brezis and Mironescu, cf. \cite{BourBrezMiro}. This approach also clarifies why the scaling in Theorem \ref{belcorped} is natural.\newline  
Let $\{\rho_{n}(x)\}_{n\in\mathbb{N}}$ be a sequence of radial mollifiers, i.e.
\begin{equation*}
\rho_n(x)=\rho_n(|x|),\ \rho_n(x)\geq0,\ \int\rho_n(x)dx=1
\end{equation*}
satisfying
\begin{equation*}
\lim\limits_{n\to\infty}\int_{\varepsilon}^{\infty}\rho_n(r)r^{N-1}=0,\ \text{for every}\ \varepsilon>0.
\end{equation*}
\begin{theorem}\label{BoBrMi}\cite[Theorem 2]{BourBrezMiro} Assume $u\in L^p(\Omega)$, $1<p<\infty$. Then, for a constant $C=C(N,p)>0$, we have
\begin{equation*}
\lim\limits_{n\to\infty}\int_{\Omega}\int_{\Omega}\frac{|u(x)-u(y)|^p}{|x-y|^p}\rho_n(x-y)dydx=C\int_{\Omega}|\nabla u|^pdx.
\end{equation*}
with the convention that $\int_{\Omega}|\nabla u|^pdx=+\infty$ if $u\notin W^{1,p}(\Omega)$.
\end{theorem}
Next, observe that we can rewrite the norm $[u]_{W^{\delta,s,p}(\Omega_{\delta})}$ as
\begin{equation*}
\begin{split}
[u]_{W^{\delta,s,p}(\Omega_{\delta})}&=\int_{\Omega_{\delta}}\int_{\Omega_{\delta}\cap B(x,\delta)}\frac{|u(x)-u(y)|^p}{|x-y|^{N+ps}}dydx\\
&=\int_{\Omega_{\delta}}\int_{\Omega_{\delta}}\frac{|u(x)-u(y)|^p}{|x-y|^{p}}\rho_{\delta}(|x-y|)dydx,
\end{split}
\end{equation*}
with $\displaystyle \rho_{\delta}(z)=\frac{\chi_{B(0,\delta)}(|z|)}{|z|^{N+p(s-1)}}$ and $\chi_A$ the characteristic function of the set $A$. In order to fulfill the hypotheses of Theorem \ref{BoBrMi} we normalize $\rho_{\delta}(z)$. Since 
\begin{equation*}
\int \rho_{\delta}(z)dz=\frac{\sigma_{N-1}}{p(1-s)}\delta^{p(1-s)},
\end{equation*}
with $\sigma_{N-1}$ the surface of the unitary sphere $\mathbb{S}^{N-1}$, the sequence of radial mollifiers
\begin{equation}\label{molliresc}
\overline{\rho}_{\delta}(z)=\frac{1}{\sigma_{N-1}}\frac{p(1-s)}{\delta^{p(1-s)}}\frac{\chi_{B(0,\delta)}(|z|)}{|z|^{N+p(s-1)}}=\frac{1}{\sigma_{N-1}}\frac{p(1-s)}{\delta^{p(1-s)}}\rho_{\delta}(z),
\end{equation}
satisfy the hypotheses of Theorem \ref{BoBrMi}. In particular, the scaling in $\delta$ coincides with the one of Theorem \ref{belcorped} and 
\begin{equation}\label{molliresc2}
\frac{1}{\sigma_{N-1}}\frac{p(1-s)}{\delta^{p(1-s)}}[u]_{W^{\delta,s,p}(\Omega_{\delta})}=\frac{1}{\sigma_{N-1}}\mathcal{E}_{\delta_n,s,p}(u)
\end{equation}
with the functional $\mathcal{E}_{\delta_n,s,p}(u)$ defined in Lemma \ref{Glimit}.
\begin{proposition}[\textit{$\Gamma$-limsup inequality}] 
Let $u\in L^p(\Omega)$ and let $\{\delta_n\}_{n\in\mathbb{N}}$ be a sequence of positive numbers such that $\delta_n\to0$ as $n\to+\infty$. Then, there exists a sequence $\{u_n\}_{n\in\mathbb{N}}\subset\mathcal{X}_0^{\delta,s,p}(\Omega)$ such that 
\begin{equation*}
\limsup_{n\to+\infty}\mathcal{E}_{\delta_n,s,p}(u_n)\leq \mathcal{E}_{0}^{1,p}(u).
\end{equation*}
\end{proposition}
\begin{proof}
If $u\notin W_0^{1,p}(\Omega)$, there is nothing to prove, then, let us take $u\in W_0^{1,p}(\Omega)$ and let $u_{\delta}$ be the function $u$ extended by zero to $\partial_{\delta}\Omega$ for $\delta>0$. If we consider the recovery sequence $\{u_{\delta_n}\}_{n\in\mathbb{N}}$ with $\delta_n\to0$ as $n\to+\infty$, because of \eqref{molliresc2}, \eqref{molliresc} and Theorem \ref{BoBrMi}, for $j\in\mathbb{N},\, j>0$ fixed, we find 
\begin{equation*}
\begin{split}
\limsup_{n\to+\infty}\mathcal{E}_{\delta_n,s,p}(u_{\delta_n})&=\limsup_{n\to+\infty}\frac{p(1-s)}{\delta_n^{p(1-s)}}\int_{\Omega_{\delta_n}}\int_{\Omega_{\delta_n}\cap B(x,\delta_n)}\frac{|u_{\delta_n}(x)-u(y)|^p}{|x-y|^{N+ps}}dydx\\
&=\sigma_{N-1}\limsup_{n\to+\infty}\int_{\Omega_{\delta_n}}\int_{\Omega_{\delta_n}}\frac{|u_{\delta_n}(x)-u_{\delta_n}(y)|^p}{|x-y|^{p}}\overline{\rho}_{\delta_n}(|x-y|)dydx\\
&\leq\sigma_{N-1}\limsup_{n\to+\infty}\int_{\Omega_{\delta_{j}}}\int_{\Omega_{\delta_{j}}}\frac{|u_{\delta_{j}}(x)-u_{\delta_{j}}(y)|^p}{|x-y|^{p}}\overline{\rho}_{\delta_n}(|x-y|)dydx\\
&=\sigma_{N-1}C(N,p)\int_{\Omega_{\delta_{j}}}|\nabla u_{\delta_{j}}|^pdx\\
&=\sigma_{N-1}C(N,p)\int_{\Omega}|\nabla u|^pdx,
\end{split}
\end{equation*}
since $u_{\delta_{j}}=0$ on $\partial_{\delta_{j}}\Omega$ and $u_{\delta_{j}}=u$ in $\Omega$. The constant $C(N,p)$ appearing in Theorem \ref{BoBrMi} takes the form, cf. \cite{BourBrezMiro},
\begin{equation*}
C(N,p)=\frac{1}{\sigma_{N-1}}\int_{\mathbb{S}^{N-1}}|z\cdot e|^pd\sigma,
\end{equation*}
for any unitary vector $e\in\mathbb{S}^{N-1}$. Therefore, $\sigma_{N-1}C(N,p)=\gamma(N,p)$ with $\gamma(N,p)$ the constant appearing in Lemma \ref{horizon0}. Hence, we conclude
\begin{equation*}
\limsup_{n\to+\infty}\mathcal{E}_{\delta_n,s,p}(u)\leq\gamma(N,p)\int_{\Omega}|\nabla u|^pdx=\mathcal{E}_{0}^{1,p}(u).
\end{equation*}
\end{proof}
\end{remark}


\section{Taking the horizon $\delta\to+\infty$}\label{horizoninfty}
In this last section, we focus on the behavior of \eqref{eigenproblem} when $\delta\to+\infty$. Let us note that, by its very definition, the operator $(-\Delta_p)_{\delta}^s$ could be viewed as a restriction of the fractional $p\,$-Laplacian in the sense that long-range interactions (beyond the horizon $\delta>0$) are neglected. Hence, if we take $\delta\to+\infty$ we recover, at least formally, the definition of the standard fractional $p\,$-Laplacian, namely
\begin{equation*}
\lim\limits_{\delta\to +\infty}(-\Delta_p)_{\delta}^s u(x)=2P.V.\int_{\mathbb{R}^N}\frac{|u(x)-u(y)|^{p-2}(u(x)-u(y))}{|x-y|^{N+ps}}dy.
\end{equation*}
Our aim is to make this rigorous by showing Theorem \ref{theigen2}. First, we prove Lemma \ref{isomorfia}. 
\begin{proof}[Proof of Lemma \ref{isomorfia}] 
Using \eqref{Gagliseminorm} and \eqref{normXcero}, we have
\begin{equation}\label{normcomparison3}
[v]_{W^{s,p}(\mathbb{R}^N)}^p-[v]_{W^{\delta,s,p}(\Omega_{\delta})}^p=\iint\limits_{\mathcal{D}\backslash\mathcal{D}_{\delta}}\frac{|v(x)-v(y)|^p}{|x-y|^{N+ps}}dydx\geq0,
\end{equation}
because $\mathcal{D}_{\delta}\subset\mathcal{D}$ for all $\delta>0$. Thus, given $v\in\mathcal{X}_0^{\infty,s,p}(\Omega)$, since $v=0$ on $\Omega^c$ we have $v=0$ on $\partial_{\delta}\Omega$ and, then, the restriction operator,
\begin{equation*}
\begin{array}{l}
     R:\mathcal{X}_0^{\infty,s,p}(\Omega)\mapsto\mathcal{X}_0^{\delta,s,p}(\Omega)\\
              \mkern+100mu v\mapsto R[v]=v\big|_{\Omega_{\delta}}         
		 \end{array}
\end{equation*}
is a continuous linear mapping. Hence, $\mathcal{X}_0^{\infty,s,p}(\Omega)$ can be continuously embedded into $\mathcal{X}_0^{\delta,s,p}(\Omega)$. On the other hand, for a given $v\in\mathcal{X}_0^{\delta,s,p}(\Omega)$, extending $v$ by 0 on $\mathbb{R}^N\backslash \Omega_{\delta}$, we have
\begin{equation*}
\begin{split}
\iint\limits_{\mathcal{D}}\frac{|v(x)-v(y)|^p}{|x-y|^{N+ps}}dydx=&\int_{\Omega_{\delta}}\int_{\Omega_{\delta}\cap B(x,\delta)}\frac{|v(x)-v(y)|^p}{|x-y|^{N+ps}}dydx\\
&+\int_{\Omega_{\delta}}\int_{\Omega_{\delta}\backslash B(x,\delta)}\frac{|v(x)-v(y)|^p}{|x-y|^{N+ps}}dydx\\
&+2\int_{\Omega_\delta}\int_{\Omega_\delta^c}\frac{|v(x)-v(y)|^p}{|x-y|^{N+ps}}dydx
\end{split}
\end{equation*}
Since the first eigenvalue $\lambda_{1}^{\delta,s,p}$ is given as 
\begin{equation*}
\lambda_{1}^{\delta,s,p}=\min\limits_{u\in\mathcal{S}^{\delta,s,p}(\Omega)}[u]_{W^{\delta,s,p}(\Omega_{\delta})}^p=\min\limits_{u\in\mathcal{X}_0^{\delta,s,p}(\Omega)}\frac{[u]_{W^{\delta,s,p}(\Omega_{\delta})}^p}{\|u\|_{L^p(\Omega)}^p},
\end{equation*}
we have 
\begin{equation*}
\begin{split}
\int_{\Omega_{\delta}}\int_{\Omega_{\delta}\backslash B(x,\delta)}\!\!\frac{|v(x)-v(y)|^p}{|x-y|^{N+ps}}dydx
\leq&\frac{1}{\delta^{N+ps}}\int_{\Omega_{\delta}}\int_{\Omega_{\delta}\backslash B(x,\delta)}\!\!|v(x)-v(y)|^pdydx\\
\leq&\frac{2^{p-1}}{\delta^{N+ps}}\int_{\Omega_{\delta}}\int_{\Omega_{\delta}\backslash B(x,\delta)}\!\!|v(x)|^p+|v(y)|^pdydx\\
\leq&\frac{2^{p-1}}{\delta^{N+ps}}\int_{\Omega_{\delta}}\int_{\Omega_{\delta}}|v(x)|^p+|v(y)|^pdydx\\
=&\frac{2^p}{\delta^{N+ps}}\int_{\Omega_{\delta}}\int_{\Omega_{\delta}}|v(x)|^pdydx\\
=&\frac{2^p|\Omega_{\delta}|}{\delta^{N+ps}}\|v\|_{L^p(\Omega)}^{p}\\
\leq&\frac{2^p|\Omega_{\delta}|}{\delta^{N+ps}}\frac{1}{\lambda_1^{\delta,s,p}}\!\int_{\Omega_{\delta}}\!\int_{\Omega_{\delta}\cap B(x,\delta)}\!\!\frac{|v(x)-v(y)|^p}{|x-y|^{N+ps}}dydx.
\end{split}
\end{equation*}
On the other hand,
\begin{equation*}
\begin{split}
\int_{\Omega_{\delta}}\int_{\Omega_{\delta}^c}\frac{|v(x)-v(y)|^p}{|x-y|^{N+ps}}dydx
=&\int_{\Omega}\int_{\Omega_{\delta}^c} \frac{|v(x)|^p}{|x-y|^{N+ps}} dydx\\
\leq&\int_{\Omega}\int_{ (B(x,\delta))^c} \frac{|v(x)|^p}{|x-y|^{N+ps}}dydx\\
=&\int_{\Omega}|v(x)|^p\int_{ (B(0,\delta))^c} \frac{1}{|z|^{N+ps}}dzdx\\
=&\frac{\sigma_{N-1}}{ps}\frac{1}{\delta^{ps}}\|v\|_{L^p(\Omega)}^{p}\\
\leq&\frac{\sigma_{N-1}}{ps}\frac{1}{\delta^{ps}}\frac{1}{\lambda_1^{\delta,s,p}}\int_{\Omega_{\delta}}\int_{\Omega_{\delta}\cap B(x,\delta)}\frac{|v(x)-v(y)|^p}{|x-y|^{N+ps}}dydx,
\end{split}
\end{equation*}
with $\sigma_{N-1}$ the surface of the unitary sphere $\mathbb{S}^{N-1}$. Therefore, 
\begin{equation*}
[v]_{W^{s,p}(\mathbb{R}^N)}\leq C(\delta)[v]_{W^{\delta,s,p}(\Omega_{\delta})}
\end{equation*}
with the constant $\displaystyle C(\delta)=\left(1+\left(\frac{2^p|\Omega_{\delta}|}{\delta^{N+ps}}+\frac{\sigma_{N-1}}{ps}\frac{1}{\delta^{ps}}\right)\frac{1}{\lambda_1^{\delta,s,p}}\right)^{\frac1p}$. As a consequence, given $v\in\mathcal{X}_0^{\delta,s,p}(\Omega)$, the extension operator
\begin{equation*}
\begin{array}{l}
E:\mathcal{X}_0^{\delta,s,p}(\Omega)\mapsto\mathcal{X}_0^{\infty,s,p}(\Omega)\\
        \mkern+92.25mu v\mapsto E[v]= \left\{\begin{array}{rl}
                                 v&\mbox{in}\quad \Omega_{\delta},\\
                                 0&\mbox{in}\quad \mathbb{R}^N\backslash\Omega_{\delta},\\
\end{array}\right.
\end{array}
\end{equation*}
is a linear continuous mapping so that $\mathcal{X}_0^{\delta,s,p}(\Omega)$ can be continuously embedded into $\mathcal{X}_0^{s,p}(\Omega)$. Next, because of \eqref{contenido} and \eqref{normcomparison3}, for any positive $\delta_1,\,\delta_2$ with $\delta_1<\delta_2$, we get
\begin{equation*}
\mathcal{X}_0^{\delta_1,s,p}(\Omega)\subset\mathcal{X}_0^{\delta_2,s,p}(\Omega)\subset \mathcal{X}_0^{s,p}(\Omega).
\end{equation*}
In particular, the sequence of first eigenvalues $\{\lambda_1^{\delta,s,p}\}_{\delta>0}$ is increasing in $\delta$ and uniformly bounded from above by the first eigenvalue $\lambda_{1}^{\infty,s,p}$. Therefore, setting $\omega_{N}$ the volume of the $N$-dimensional unitary ball, we have
\begin{equation*}
\begin{split}
C(\delta)&=\left(1+\left(\frac{2^p|\Omega_{\delta}|}{\delta^{N+ps}}+\frac{\sigma_{N-1}}{ps}\frac{1}{\delta^{ps}}\right)\frac{1}{\lambda_1^{\delta,s,p}}\right)^{\frac1p}\\
&\leq\left(1+\left(\frac{2^p\omega_{N}(diam(\Omega)+2\delta)^N}{\delta^{N+ps}}+\frac{\sigma_{N-1}}{ps}\frac{1}{\delta^{ps}}\right)\frac{1}{\lambda_1^{\delta,s,p}}\right)^{\frac1p}\\
&=1+O\left(\frac{1}{\delta^s}\right).
\end{split}
\end{equation*} 
Thus, we conclude $\lim\limits_{\delta\to+\infty}C(\delta)=1$.
\end{proof}

The following $\Gamma$-convergence result is in the core of the proof of Theorem \ref{theigen2}.
\begin{lemma}\label{Gconvergenceinf}Let us consider the functional 
\begin{equation*}
\mathcal{J}_{\delta,s,p}(u)=[u]_{W^{\delta,s,p}(\Omega_{\delta})}^p,
\end{equation*} 
defined on $\mathcal{X}_0^{\delta,s,p}(\Omega)$. Then, the $\Gamma$-limit of $\mathcal{J}_{\delta,s,p}$ as $\delta\to+\infty$ is given by
\begin{equation}\label{GlimitFuncinf}
\mathcal{J}_{\infty,s,p}(u)\vcentcolon=[u]_{W^{s,p}(\mathbb{R}^N)}^p.
\end{equation}
\end{lemma}
\begin{proof}
The sequence of functionals $\{\mathcal{J}_{\delta,s,p}\}_{\delta>0}$ is a monotone increasing sequence as $\delta\to+\infty$. Moreover, for any $\delta>0$, the functional $\mathcal{J}_{\delta,s,p}$ is lower semicontinuous, cf. \cite{BellCor}. Therefore, because of \cite[Remark 1.40]{Braides}, 

\begin{equation*}
\left(\Gamma-\lim\limits_{\delta\to+\infty}\mathcal{J}_{\delta,s,p}\right)(u)=\mathcal{J}_{\infty,s,p}(u)=[u]_{W^{s,p}(\mathbb{R}^N)}^p,
\end{equation*}
for all $u\in L^p(\Omega)$.
\end{proof}

\begin{proof}[Proof of Theorem \ref{theigen2}]
The proof follows by combining Lemma \ref{Gconvergenceinf} and Lemma $\ref{isomorfia}$ with the arguments used in the proof of Theorem \ref{theigen}.\newline 
Let $\varphi_k^{\delta,s,p}$ be a minimizer of $\mathcal{J}_{\delta,s,p}$ such that $\|\varphi_k^{\delta,s,p}\|_{L^p(\Omega)}=1$. Observe that the sequence $\{\mathcal{J}_{\delta,s,p}(\varphi_k^{\delta,s,p})\}_{\delta>0}$ is monotone increasing in $\delta>0$ and bounded from above by $\mathcal{J}_{\infty,s,p}(\varphi_k^{\infty,s,p})$. Indeed, given $0<\delta_1<\delta_2$, we have 
\begin{equation*} 
\mathcal{J}_{\delta_1,s,p}(\varphi_k^{\delta_1,s,p})\leq \mathcal{J}_{\delta_1,s,p}(\varphi_k^{\delta_2,s,p}) \le \mathcal{J}_{\delta_2,s,p}(\varphi_k^{\delta_2,s,p}).
\end{equation*} 
Thus, because of Theorem \ref{belcorped}-a), there exists a subsequence $\{\varphi_k^{\delta_n,s,p}\}_{n\in\mathbb{N}}$ and a function $\varphi_k^{\infty,s,p}\in \mathcal{X}_0^{\infty,s,p}(\Omega)$, such that
\begin{equation*} 
\varphi_k^{\delta_n,s,p}\to \varphi_k^{\infty,s,p}\ \text{ strong in }L^p(\Omega)\quad\text{as }n\to+\infty.
\end{equation*}
The rest follows similarly to the proof of Theorem \ref{theigen}.
\end{proof}

\end{document}